\documentclass[12pt]{article}

\usepackage{amsmath,amssymb,amsthm}
\usepackage[english]{babel}
\usepackage{enumitem}
\usepackage{varwidth}
\usepackage{geometry}
\usepackage{wasysym}

\newtheorem{theorem}{Theorem}[subsection]
\newtheorem{lemma}[theorem]{Lemma}
\newtheorem{proposition}[theorem]{Proposition}
\newtheorem{corollary}[theorem]{Corollary}
\theoremstyle{definition}
\newtheorem{definition}[theorem]{Definition}
\newtheorem{example}[theorem]{Example}

\theoremstyle{remark}

\newcommand{\NN}{\mathbb{N}}

\newcommand{\restr}[2]{{#1}_{|#2}}
\newcommand{\supp}[1]{{\textnormal{supp}(#1)}}
\newcommand{\wt}[1]{{\textnormal{wt}(#1)}}
\newcommand{\spann}[2]{{\langle#1,#2\rangle}}
\newcommand{\zero}{{T_0}}
\newcommand{\codline}{{T_{q+1}}}
\newcommand{\difference}{{T_{2q}}}
\newcommand{\twolines}{{T_{2q+1}}}
\newcommand{\odd}{{T^{\textnormal{odd}}}}
\newcommand{\nonconcurrent}{{T^{\boldsymbol{\triangle}}}}
\newcommand{\concurrent}{{T^{\bigstar}}}

\title{Small weight code words arising from the incidence of points and hyperplanes in PG($n,q$)}
\author{Sam Adriaensen, Lins Denaux, Leo Storme \& Zsuzsa Weiner\thanks{The fourth author acknowledges the support of OTKA Grant No. K 124950.}}
\date{}

\begin{document}
	
	\maketitle
	
	\begin{abstract}
		Let $C_{n-1}(n,q)$ be the code arising from the incidence of points and hyperplanes in the Desarguesian projective space PG($n,q$). Recently, Polverino and Zullo \cite{polverino} proved that within this code, all non-zero code words of weight at most $2q^{n-1}$ are scalar multiples of either the incidence vector of one hyperplane, or the difference of the incidence vectors of two distinct hyperplanes. We improve this result, proving that when $q>17$ and $q\notin\{25,27,29,31,32,49,121\}$, all code words of weight at most $(4q-\sqrt{8q}-\frac{33}{2})q^{n-2}$ are linear combinations of incidence vectors of hyperplanes through a fixed $(n-3)$-space. Depending on the omitted value for $q$, we can lower the bound on the weight of $c$ to obtain the same results.
	\end{abstract}
	
	\section{Preliminaries}
	
	Let $n\in\NN$ and $q=p^h$, with $p$ prime and $h\in\NN\setminus\{0\}$. Let PG($n,q$) be the $n$-dimensional Desarguesian projective space over the finite field of order $q$. In line with other articles, we define
	\[
		\theta_n=\frac{q^{n+1}-1}{q-1}\textnormal{,}
	\]
	with the extension that $\theta_m=0$ if $m\in\mathbb{Z}\setminus\NN$. Denote the set of all points of PG($n,q$) by $\mathcal{P}(n,q)$ and the set of all hyperplanes by $\mathcal{H}(n,q)$. Let $\mathcal V(n,q)$ be the $p$-ary vector space of functions from $\mathcal{P}(n,q)$ to $\mathbb{F}_p$; thus $\mathcal V(n,q)=\mathbb{F}_p^{\mathcal{P}(n,q)}$. Denote by $\boldsymbol{1}$ the function that maps all points to $1$.
	
	\begin{definition}
	    Let $v\in \mathcal V(n,q)$. Define the \emph{support} of $v$ as $\supp{v}=\{P\in\mathcal{P}(n,q):v(P)\neq0\}$ and the \emph{weight} of $v$ as $\wt{v}=|\supp{v}|$. We will call all points of $\mathcal{P}(n,q)\setminus\supp{v}$ the \emph{holes} of $v$.
	\end{definition}
	
	We can identify each hyperplane $H\in\mathcal{H}(n,q)$ with the function $H\in \mathcal V(n,q)$ such that
	\[
	    H(P)=\begin{cases}1\quad\textnormal{if }P\in H\textnormal{,}\\
	    0\quad\textnormal{otherwise.}\end{cases}
	\]
	If a hyperplane $H$ is identified as a function, its representation as a vector will be called the \emph{incidence vector} of the hyperplane $H$. It should be clear from the context whether we mean an actual hyperplane or such a function/vector.
	
	\begin{definition}
		The $p$-ary linear code $C_{n-1}(n,q)$ is the subspace of $\mathcal V(n,q)$ generated by $\mathcal{H}(n,q)$, where we interpret the elements of the latter as functions in $\mathcal V(n,q)$. The elements of $C_{n-1}(n,q)$ are called $\emph{code words}$.
	\end{definition}
	
	Define the \emph{scalar product} of two functions $v,w\in \mathcal V(n,q)$ as
	\[
	    v\cdot w=\sum_{P\in\mathcal{P}(n,q)}v(P)w(P)\textnormal{.}
	\]
	
	\begin{definition}
		We define the \emph{dual code} of ${C_{n-1}(n,q)}$ as its orthogonal complement with respect to the above scalar product:
		\[
		    {C_{n-1}(n,q)}^\bot=\big\{v\in \mathcal V(n,q):\big(\forall c\in C_{n-1}(n,q)\big)\big(c\cdot v=0\big)\big\}\textnormal{.}
		\]
	\end{definition}
	
	\begin{definition}
		Let $v\in \mathcal V(n,q)$ and take a $k$-space $\kappa$ in PG($n,q$). If we let $\kappa$ play the role of PG($k,q$), we can naturally define the \emph{restriction} of $v$ to the space $\kappa$ as the function $\restr{v}{\kappa}\in \mathcal V(k,q)$ restricted to the point set $\mathcal{P}(k,q)\subseteq \mathcal{P}(n,q)$.
	\end{definition}

	\begin{definition}
		Let $s$ be a line in PG($n,q$) and $v\in \mathcal V(n,q)$. If $s$ intersects $\supp{v}$ in $\alpha$ points ($0\leqslant\alpha\leqslant q+1$), we will call $s$ an $\alpha$\emph{-secant} to $\supp{v}$. Furthermore,
		\begin{itemize}
			\item if $\alpha\leqslant3$, $s$ will be called a \emph{short secant},
			\item if $\alpha\geqslant q-1$, $s$ will be called a \emph{long secant}.
		\end{itemize}
	\end{definition}
	
	\section{Known results}
	
	\subsection{Results in general dimension}
	
	The minimum weight of the code $C_{n-1}(n,q)$ equals $\theta_{n-1}$. The code words corresponding to this weight are characterised.
	
	\begin{theorem}[\cite{assmuskey,macwilliams}]\label{MinimumWeight}
		The code words of $C_{n-1}(n,q)$ having minimum weight are the scalar multiples of the incidence vectors of hyperplanes.
	\end{theorem}

	Bagchi and Inamdar \cite[Theorem 1]{bagchiminweight} gave a geometrical proof of this theorem, using \emph{blocking sets}. Recently, Polverino and Zullo \cite{polverino} characterised all code words up to the second smallest (non-zero) weight:
	
	\begin{theorem}[{\cite[4]{polverino}}]\label{ResultsPolverinoZullo}
		Let $q = p^h$ with $p$ prime.
		\begin{enumerate}
			\item There are no code words of $C_{n-1}(n,q)$ with weight in the interval $]\theta_{n-1},2q^{n-1}[$.
			\item The code words of weight $2q^{n-1}$ in $C_{n-1}(n,q)$ are the scalar multiples of the difference of the incidence vectors of two distinct hyperplanes of PG($n,q$).
		\end{enumerate}
	\end{theorem}

	So far, Theorem \ref{ResultsPolverinoZullo} summarises the best results known concerning the characterisation of small weight code words in $C_{n-1}(n,q)$ in case $n\geqslant3$.
	
	\bigskip
	As a final note, we keep the following lemmata in mind.
	
	\begin{lemma}[{\cite[Chapter $6$]{assmuskey}, \cite[Lemma $2$]{polverino}}]\label{SumOfCoefficients}
	    Let $c\in C_{n-1}(n,q)$, $c=\sum_i\alpha_iH_i$ for some $\alpha_i\in\mathbb{F}_p\setminus\{0\}$ and $H_i\in\mathcal{H}(n,q)$, and let $\kappa$ be a $k$-space of PG($n,q$), $1\leqslant\kappa\leqslant n$. Then $\kappa\cdot c=\sum_i\alpha_i$.
	\end{lemma}
	
	\begin{lemma}[{\cite[Remark 3.1]{polverino}}]
	\label{InductiveLemma}
		Let $c\in C_{n-1}(n,q)$ be a code word and $\kappa$ a $k$-space of PG($n,q$), $1\leqslant k\leqslant n$. Then $\restr{c}{\kappa}$ is a code word of $C_{k-1}(k,q)$.
	\end{lemma}
	
	\subsection{Results in the plane}
	
	Historically, most of the work done on this topic focuses on the planar case, i.e. the code of points and lines $C_1(2,q)$. Some early results on small weight code words in this particular code were those of Chouinard. In his PhD Thesis \cite{chouinardprime}, he proved that, when $q=p$ prime, code words up to weight $2p$ are linear combinations of at most two lines. When $q=9$, he proved that code words having a weight in the interval $]q+1,2q[$ do not exist \cite{chouinard9}.\\
	Fack et al.\ \cite{fack} improved the prime case. More specifically, these authors proved that, if $q=p\geqslant11$, all code words of weight up to $2p+\frac{p-1}{2}$ are linear combinations of at most two lines. They cleverly made use of the existence of a Moorhouse base \cite{moorhouse}.\\
	The prime case kept on inspiring more mathematicians. Next in line is Bagchi \cite{bagchiweight} on the one hand, and Sz\H onyi and Weiner \cite{szonyi} on the other hand (see Theorem \ref{ResultsSzonyiWeinerPrime}). Bagchi proved the following:

	\begin{theorem}[{\cite[Theorem $1.1$]{bagchiweight}}]\label{ResultsBagchi}
		Let $p\geqslant5$. Then, the fourth smallest weight of $C_1(2,p)$ is $3p-3$. The only words of $C_1(2,p)$ of Hamming weight smaller than $3p-3$ are the linear combinations of at most two lines in the plane.
	\end{theorem}

	Bagchi knew this bound was sharp, as he discovered a code word of weight $3p-3$ which \emph{cannot} be constructed as a linear combination of at most two lines when $p>3$ \cite{bagchicodeword}. This code word was independently discovered by De Boeck and Vandendriessche \cite{deboeck} as well.
	
	\begin{example}[{\cite{bagchicodeword},\cite[Example $10.3.4$]{deboeck}}]\label{OddCodeword}
		Choose a coordinate system for PG($2,p$) and let $c$ be a vector of $\mathcal V(2,p)$, $p\neq2$ a prime, such that
		\[
			c(P)=\begin{cases}
				a\quad&\textnormal{if}\quad P=(0,1,a)\textnormal{,}\\
				b\quad&\textnormal{if}\quad P=(1,0,b)\textnormal{,}\\
				-c\quad&\textnormal{if}\quad P=(1,1,c)\textnormal{,}\\
				0\quad&\textnormal{otherwise.}
			\end{cases}
		\]
		Remark that $\supp{c}$ is covered by the three concurrent lines $m: X_0=0$, $m': X_1=0$ and $m'': X_0=X_1$.
	\end{example}

	The proof of $c$ being a code word of $C_1(2,p)$ relies on proving that $c$ belongs to ${C_1(2,p)}^\bot\subseteq C_1(2,p)$. As each of the three lines $m$, $m'$ and $m''$ contains $p-1$ points with pairwise different, non-zero values, it is easy to see that such a code word can never be written as a linear combination of less than $p-1$ different lines.\\
	As noted by Sz\H onyi and Weiner \cite{szonyi}, the above example can be generalised as follows:
	
	\begin{example}({\cite[Example $4.7$]{szonyi}})\label{OddCodewordGen}
		Let $c$ be the code word in Example \ref{OddCodeword}, with corresponding lines $m$, $m'$ and $m''$ considered as incidence vectors. Suppose $\pi$ is an arbitrary collineation of PG($2,p$) and let $\gamma\in\mathbb{F}_p\setminus\{0\}$ and $\lambda,\lambda',\lambda''\in\mathbb{F}_p$. Then
		\[
			d=(\gamma c+\lambda m+\lambda'm'+\lambda''m'')^\pi
		\]
		is a code word of weight $3p-3$ or $3p-2$, depending on the value of $\lambda+\lambda'+\lambda''$.
	\end{example}

	By construction, it is easy to see that this generalised example has some interesting properties.

	\begin{proposition}\label{PropOddCodeword}
		Suppose $d$ is the code word as constructed in Example \ref{OddCodewordGen}. Let $S=(m\cap m'\cap m'')^\pi$. Then
		\[
			\wt{d}=\begin{cases}
				3p-3\quad\textnormal{if}\quad d(S)=0\textnormal{,}\\
				3p-2\quad\textnormal{if}\quad d(S)\neq0\textnormal{.}
			\end{cases}
		\]
	\end{proposition}

	For somewhat larger values of $p$, Sz\H onyi and Weiner \cite{szonyi} improved Bagchi's result:

	\begin{theorem}[{\cite[Theorem $4.8$ and Corollary $4.10$]{szonyi}}]\label{ResultsSzonyiWeinerPrime}
		Let $c$ be a code word of $C_1(2,p)$, $p>17$ prime. If $\wt{c}\leqslant\max\{3p+1,4p-22\}$, then $c$ is either the linear combination of at most three lines or given by Example \ref{OddCodewordGen}.
	\end{theorem}

	The same authors have proven the following results for $q$ not prime, proving for large values of $q$ that the code word described in Example \ref{OddCodewordGen} can only exist when $q$ is prime.

	\begin{theorem}[{\cite[Theorem $4.3$]{szonyi}}]\label{ResultsSzonyiWeinerNonPrime}
		Let $c$ be a code word of $C_1(2,q)$, with $27 < q$, $q = p^h$, $p$ prime. If
		\begin{itemize}
			\item $\wt{c}<(\lfloor\sqrt{q}\rfloor+1)(q+1-\lfloor\sqrt{q}\rfloor)$, when $2<h$, or
			\item $\wt{c}<\frac{(p-1)(p-4)(p^2+1)}{2p-1}$, when $h=2$,
		\end{itemize}
		then $c$ is a linear combination of exactly $\big\lceil\frac{\wt{c}}{q+1}\big\rceil$ different lines.
	\end{theorem}

	We can now summarise these results concerning $C_1(2,q)$ in one corollary:
	
	\begin{corollary}\label{PlaneResults}
		Let $c$ be a code word of $C_1(2,q)$, with $q=p^h$, $p$ prime, and $q\notin\{8,9,16,25,27,49\}$.
		\begin{itemize}
			\item If $\wt{c}\leqslant3q-4$, then $c$ is a linear combination of at most two lines.
			\item If $\wt{c}\leqslant3q+1$ and $q=121$, then $c$ is a linear combination of at most three lines.
			\item If $\wt{c}\leqslant\max\{3q+1,4q-22\}$ and $q>17$, $q\neq121$, then $c$ is a linear combination of at most three lines or given by Example \ref{OddCodewordGen}.
		\end{itemize}
	\end{corollary}
	\begin{proof}
		If $q\leqslant4$, then $3q-4\leqslant2q$ and we can use Theorem \ref{ResultsPolverinoZullo}. If $q>4$ and $q$ is prime, the proof immediately follows from Theorem \ref{ResultsBagchi} and Theorem \ref{ResultsSzonyiWeinerPrime}.\\
		Suppose $q>4$ is not prime. Then, by assumption, $q>27$, which means that $\max\{3q+1,4q-22\}=4q-22$. To apply Theorem \ref{ResultsSzonyiWeinerNonPrime}, we only have to check the weight assumptions. One can verify that
		\begin{itemize}
			\item $4q-22 < (\lfloor\sqrt{q}\rfloor+1)(q+1-\lfloor\sqrt{q}\rfloor)$ if $q\geqslant10$,
			\item $3p^2+1 < \frac{(p-1)(p-4)(p^2+1)}{2p-1}$ if $q=p^2\geqslant121$,
			\item $4p^2-22 < \frac{(p-1)(p-4)(p^2+1)}{2p-1}$ if $q=p^2\geqslant144$.
		\end{itemize}
		We conclude that $c$ is a linear combination of at most $\lceil\frac{4q-22}{q+1}\rceil = 4$ lines. If $c$ is a linear combination of precisely $4$ lines, then its weight is at least $4\cdot\big((q+1)-3\big)=4q-8$, a contradiction.
	\end{proof}

	\section{The main theorem}
	
	Throughout this section, let $n\in\mathbb{N}\setminus\{0,1\}$ and $q=p^h$, with $p$ prime and $h\in\mathbb{N}\setminus\{0\}$. Let $c\in C_{n-1}(n,q)$ be an arbitrary code word. Furthermore, define
	\[
		B_{n,q}=\begin{cases}
		    \qquad\qquad2q^{n-1}\quad&\textnormal{if }q<7\textnormal{ or }q\in\{8,9,16,25,27,49\}\textnormal{,}\\
		    \Big(3q-\sqrt{6q}-\frac{1}{2}\Big)q^{n-2}\quad&\textnormal{if }q\in\{7,11,13,17\}\textnormal{,}\\
		    \Big(3q-\sqrt{6q}+\frac{9}{2}\Big)q^{n-2}\quad&\textnormal{if }q\in\{19,121\}\textnormal{,}\\
		    \Big(4q-4\sqrt{q}-\frac{25}{2}\Big)q^{n-2}\quad&\textnormal{if }q\in\{29,31,32\}\textnormal{,}\\
		    \Big(4q-\sqrt{8q}-\frac{33}{2}\Big)q^{n-2}\quad&\textnormal{otherwise.}\\
		\end{cases}
	\]
	This will be the assumed upper bound on the weight of $c$. By Theorem \ref{ResultsPolverinoZullo}, we can always assume that $q\geqslant7$ and $q\notin\{8,9,16,25,27,49\}$.
	
	\subsection{Preliminaries}
	
	Using Corollary \ref{PlaneResults} and Lemma \ref{InductiveLemma}, we can distinguish several types of small weight code words:
	
	\begin{definition}\label{DefPlane}
	Let $c$ be a code word, and $\pi$ a plane. We will call $\restr{c}{\pi}$
		\begin{itemize}
			\item a \emph{code word of type} $T_w$ if $\restr{c}{\pi}$ is a linear combination of at most two lines, with $w$ the weight of $\restr{c}{\pi}$.
			\item a \emph{code word of type} $\odd$ if $\restr{c}{\pi}$ is a code word as described in Example \ref{OddCodewordGen}.
			\item a \emph{code word of type} $\nonconcurrent$ if $\restr{c}{\pi}$ is a linear combination of three nonconcurrent lines.
			\item a \emph{code word of type} $\concurrent$ if $\restr{c}{\pi}$ is a linear combination of three concurrent lines.
			\item a \emph{code word of type} $\boldsymbol{\mathcal{T}}=\{\zero,\codline,\difference,\twolines,\odd,\nonconcurrent,\concurrent\}$ if $\restr{c}{\pi}$ is a code word of one of the types mentioned above.
			\item a \emph{code word of type} $\boldsymbol{\mathcal{O}}$ if $\restr{c}{\pi}$ is \emph{not} a code word of one of the types mentioned above.
		\end{itemize}
		 We will often make no distinction between the code word $\restr{c}{\pi}$ and the plane $\pi$: if $\restr{c}{\pi}$ is a code word of a certain type $T$, we will call $\pi$ a \emph{plane of type} $T$.
	\end{definition}

	\begin{proposition}\label{LinesIncharacterisedPlanes}
		If $\pi$ is a plane of type $\boldsymbol{\mathcal{T}}$ in PG($n,q$), then all lines of $\pi$ are either short or long secants to $\supp{c}$. If the type of $\pi$ is an element of $\{T_0,T_{q+1},T_{2q},T_{2q+1}\}$, then all lines intersect $\supp{c}$ in at most $2$ or in at least $q$ points.\qed
	\end{proposition}
	
	The following is a generalisation of Definition \ref{DefPlane} to arbitrary dimension.

	\begin{definition}\label{DefHyperplane}
		Let $\Pi$ be a $k$-space of PG($n,q$), $2\leqslant k\leqslant n$. We will call the code word $\restr{c}{\Pi}$ \emph{a code word of type} $T\in\boldsymbol{\mathcal{T}}$ if the following is true:
		\begin{itemize}
			\item there exists a $(k-3)$-space $\kappa$ in $\Pi$ such that $\restr{c}{\kappa}$ is a scalar multiple of $\boldsymbol{1}$.
			\item there exists a plane $\pi$ in $\Pi$ of type $T$, disjoint to $\kappa$.
			\item for all points $P\in\Pi\setminus\kappa$, $c(P)=c\big(\spann{\kappa}{P}\cap\pi\big)$.
		\end{itemize}
		If $\restr{c}{\Pi}$ is a code word of type $T$, we will often call the space $\Pi$ a space of type $T$ as well. If the type $T\in\boldsymbol{\mathcal{T}}$ is not known, we will call the code word $\restr{c}{\Pi}$ (or the space $\Pi$) \emph{of type} $\boldsymbol{\mathcal{T}}$. Remark that, if $k=2$, the above definition coincides with Definition \ref{DefPlane}.
	\end{definition}
	
	The upcoming theorem is the main theorem of this article, which is an improvement of Theorem \ref{ResultsPolverinoZullo} when $q\geqslant7$ and $q\notin\{8,9,16,25,27,49\}$:

	\begin{theorem}\label{THETHEOREM}
		Let $c$ be a code word of $C_{n-1}(n,q)$, with $n\geqslant3$, $q$ a prime power and $\wt{c}\leqslant B_{n,q}$. Then $c$ can be written as a linear combination of incidence vectors of hyperplanes through a fixed $(n-3)$-space.
	\end{theorem}

	\begin{corollary}\label{THETHEOREMCOR}
		Let $c$ be a code word of $C_{n-1}(n,q)$, with $n\geqslant3$, $q$ a prime power and $\wt{c}\leqslant B_{n,q}$. Then $c$ is a code word of type $\boldsymbol{\mathcal{T}}$.
	\end{corollary}
	\begin{proof}
	    By Theorem \ref{THETHEOREM}, $c$ can be written as a linear combination of the incidence vectors of hyperplanes through a fixed $(n-3)$-space $\kappa$. If all planes disjoint to $\kappa$ are planes of type $\boldsymbol{\mathcal{O}}$, Lemma \ref{LinesAndPlanes} implies that $\wt{c}\geqslant\frac{1}{2}q^{n-2}(q^2-3q-2)$, a contradiction for all $q$. All other properties of Definition \ref{DefHyperplane} can easily be checked.
	\end{proof}
	
	\subsection{The proof}
	
	\subsubsection{The weight spectrum concerning lines and planes}
	
	In this subsection, we will prove some intermediate results, the first one stating that all lines contain few or many points of $\supp{c}$.
	
	\begin{lemma}\label{LinesAndPlanes}
		Suppose $\wt{c}\leqslant B_{n,q}$. Then
		\begin{enumerate}
		    \item all lines are either short or long secants to $\supp{c}$. If $q\leqslant17$, then all lines intersect $\supp{c}$ in at most $2$ or in at least $q$ points.
		    \item all planes of type $\boldsymbol{\mathcal{O}}$ contain at least $\frac{1}{2}q(q-1)$ points of $\supp{c}$.
		\end{enumerate}
	\end{lemma}
	\begin{proof}
	     First of all, define the values
	     \[
		    A_q=\begin{cases}
		        3q-3\quad&\textnormal{if }q\in\{7,11,13,17\}\textnormal{,}\\
    		    3q+2\quad&\textnormal{if }q\in\{19,121\}\textnormal{,}\\
    		    4q-21\quad&\textnormal{otherwise,}\\
		    \end{cases}
		    \quad\textnormal{and}\quad
		    i=\begin{cases}
		        3\quad&\textnormal{if }q\in\{7,11,13,17\}\textnormal{,}\\
    		    4\quad&\textnormal{otherwise.}\\
		    \end{cases}
		\]
	    Suppose, on the contrary, that $s$ is an $m$-secant to $\supp{c}$, with $i\leqslant m\leqslant q+2-i$. By Proposition \ref{LinesIncharacterisedPlanes} and Corollary \ref{PlaneResults}, all planes through $s$ contain at least $A_q$ points of $\supp{c}$. We get
		\begin{align}
			\wt{c}&\geqslant A_q\theta_{n-2}-m(\theta_{n-2}-1)\nonumber\\
			\Leftrightarrow m&\geqslant \frac{A_q\theta_{n-2}-\wt{c}}{\theta_{n-2}-1}\label{LowerBoundM}\\
			\Rightarrow q+2-i&\geqslant\frac{A_q\theta_{n-2}-\wt{c}}{\theta_{n-2}-1}\textnormal{.}\label{UpperBoundOnLowerBound}
		\end{align}
		From \eqref{LowerBoundM} and \eqref{UpperBoundOnLowerBound}, we can conclude that all lines intersect $\supp{c}$ in at most $i-1$ or in at least $\frac{A_q\theta_{n-2}-\wt{c}}{\theta_{n-2}-1}$ points. Let $\pi\supseteq s$ be an arbitrary plane, thus $\wt{\restr{c}{\pi}}\geqslant A_q$, and define $j=\min{\{\wt{\restr{c}{l}}:l\subseteq\pi,\wt{\restr{c}{l}}\geqslant i\}}$. Choose a point $P\in s\cap\supp{c}$. If all other $q$ lines in $\pi$ through $P$ contain at most $i-1$ points of $\supp{c}$, then $\wt{\restr{c}{\pi}}\leqslant (i-2)q+m\leqslant(i-1)q-1<A_q$, a contradiction. Thus, through each point on $s\cap\supp{c}$ we find at least one line in $\pi$, other than $s$, containing at least $j$ points of $\supp{c}$. We find at least $m\geqslant j$ such lines, meaning that
		\begin{equation}\label{TelescopicSum}
		    \wt{\restr{c}{\pi}}\geqslant j + (j-1) + \dots + 1 = \frac{1}{2}j(j+1)\textnormal{.}
		\end{equation}
		This holds for all planes through an $m$-secant with $i\leqslant m\leqslant q+2-i$, in particular for all planes through a $j$-secant in $\pi$. As such, we get
		\[
		    \wt{c}\geqslant\bigg(\frac{1}{2}j(j+1)-j\bigg)\theta_{n-2}+j\textnormal{.}
		\]
		When combining this result with $j\geqslant\frac{A_q\theta_{n-2}-\wt{c}}{\theta_{n-2}-1}$, we get a condition on $\wt{c}$, eventually leading to $\wt{c}>B_{n,q}$, a contradiction. We refer to Appendix \ref{Appendix} for the arithmetic details.\\
		Let $\pi$ be a plane of type $\boldsymbol{\mathcal{O}}$. If no long secant is contained in this plane, $\wt{\restr{c}{\pi}}\leqslant2q+1<A_q$, a contradiction. Repeating the previous arguments, we get the same result as \eqref{TelescopicSum}, for $j\geqslant q-1$. This concludes the proof.
	\end{proof}
	
	Using this result, we can deduce the following.

	\begin{lemma}\label{q-1Implies3}
		Suppose $\wt{c}\leqslant B_{n,q}$. If there exists a $(q-1)$-secant to $\supp{c}$, then there exists a $3$-secant to $\supp{c}$ as well.
	\end{lemma}
	\begin{proof}
		Let $s$ be a $(q-1)$-secant and suppose, on the contrary, that no $3$-secants exist. Remark that planes of type $\boldsymbol{\mathcal{T}}$ containing a $(q-1)$-secant always contain a $3$-secant. Hence, by Lemma \ref{LinesAndPlanes}, all planes through $s$ contain at least $\frac{1}{2}q(q-1)$ points of $\supp{c}$. We get
		\[
			B_{n,q}\geqslant\wt{c}\geqslant\bigg(\frac{1}{2}q(q-1)-(q-1)\bigg)\theta_{n-2}+(q-1)\textnormal{,}
		\]
		which is a contradiction for all values of $q$.
	\end{proof}

	\begin{lemma}\label{Lines2AndQ}
		Suppose $\wt{c}\leqslant\min\big\{(3q-6)\theta_{n-2}+2,B_{n,q}\big\}$. Then all lines intersect $\supp{c}$ in at most $2$ or in at least $q$ points.
	\end{lemma}
	\begin{proof}
		By Lemma \ref{LinesAndPlanes} and Lemma \ref{q-1Implies3}, it suffices to prove that $3$-secants to $\supp{c}$ cannot exist. Suppose there exists a $3$-secant to $\supp{c}$. By Corollary \ref{PlaneResults}, all planes containing this $3$-secant have at least $3q-3$ points in common with $\supp{c}$. This gives us the following contradiction:
		\[
			(3q-6)\theta_{n-2}+2\geqslant \wt{c}\geqslant(3q-3-3)\theta_{n-2}+3\textnormal{.}\qedhere
		\]
	\end{proof}

	\subsubsection{Code words of weight $\boldsymbol{2q^{n-1}+\theta_{n-2}}$}\label{SecThirdWeight}

	In this section, we will prove Theorem \ref{TwoHyperplaneTheorem}, which essentially states that, if $\wt{c}\leqslant\min\big\{(3q-6)\theta_{n-2}+2,B_{n,q}\big\}$, the code word $c$ corresponds to a linear combination of at most \emph{two} hyperplanes.

    \begin{lemma}\label{ClassificationOf012QQ+1Secants}
        Assume that $S$ is a point set in PG($n,q$), $q \geqslant 7$, and every line intersects $S$ in at most $2$ or in at least $q$ points.
        Then one of the following holds:
        \begin{enumerate}
            \item[\textnormal{(1)}] $|S| \leqslant 2q^{n-1} + \theta_{n-2}$.
            \item[\textnormal{(2)}] The complement of $S$, denoted by $S^c$, is contained in a hyperplane.
        \end{enumerate}
        \end{lemma}
        
        \begin{proof}
        We prove this by induction on $n$.
        Note that the statement is trivial for $n=1$, so assume that $n \geqslant 2$. Furthermore, we can inductively assume that for every hyperplane $\Pi$, either $|S \cap \Pi| \leqslant 2q^{n-2} + \theta_{n-3}$, in which case we call $\Pi$ a \emph{small hyperplane}, or $S^c \cap \Pi$ is contained in an $(n-2)$-subspace of $\Pi$, in which case we call $\Pi$ a \emph{large hyperplane}.\\
        
        \underline{Case $1$: There exist two large hyperplanes $\Pi_1$ and $\Pi_2$, and a point $P \in S \setminus (\Pi_1 \cup \Pi_2)$.}
        
        Consider the lines through $P$.
        At most $q^{n-2}$ of these lines intersect $\Pi_i \setminus \Pi_{3-i}$ in a point of $S^c$, and $\theta_{n-2}$ of these lines intersect $\Pi_1 \cap \Pi_2$.
        Hence, at least $\theta_{n-1} - 2q^{n-2} - \theta_{n-2} = q^{n-1} - 2q^{n-2}$ of these lines intersect $\Pi_1$ and $\Pi_2$ in distinct points of $S$.
        As $P \in S$, each of these lines contains at least three points of $S$.
        Therefore, they must contain at least $q$ points of $S$, thus at least $q-3$ points of $S \setminus\big(\Pi_1 \cup \Pi_2 \cup \{P\}\big)$.
        Since $\Pi_1 \cup \Pi_2$ contains at least $2q^{n-1}-q^{n-2}$ points of $S$, we know that
        \[
            |S| \geqslant (q^{n-1}-2q^{n-2})(q-3) + (2q^{n-1} - q^{n-2}) + 1 = q^n - 3q^{n-1} + 5q^{n-2} + 1.
        \]
        
        \underline{Case $2$: There exists a small hyperplane $\Pi$ and a point $P \in S^c \setminus\Pi$.}
        
        The small hyperplane $\Pi$ must contain at least $\theta_{n-1} - (2q^{n-2} + \theta_{n-3}) = q^{n-1} - q^{n-2}$ points of $S^c$.
        Every line through $P$ and a point of $\Pi \cap S^c$ intersects $S^c$ in at least 2, thus in at least $q-1$ points of $S^c$.
        This yields that
        \[
            |S^c| \geqslant (q^{n-1} - q^{n-2})(q-2) + 1 = q^n - 3q^{n-1} + 2q^{n-2} + 1.
        \]
        If both cases would occur simultaneously, then
        \begin{align*}
            \theta_n  = |S| + |S^c| 
            & \geqslant (q^n - 3q^{n-1} + 5q^{n-2} + 1) + (q^n - 3q^{n-1} + 2q^{n-2} + 1) \\
            & = 2q^n - 6q^{n-1} + 7q^{n-2} + 2 \textnormal{,}
        \end{align*}
        which is a contradiction if $q \geqslant 7$.
        Note that the existence of three large hyperplanes implies Case $1$, and the existence of two small hyperplanes implies Case $2$.
        Therefore, exactly one of these cases holds.
        
        Assume that Case $1$ holds. Hence, Case $2$ cannot hold, so if there exists a small hyperplane, it has to contain the entirety of $S^c$ and the proof is done. As such, we can assume that all hyperplanes are large.
        Take a hyperplane $\Pi$. If the points of $S^c \cap \Pi$ span an $(n-2)$-space $\Sigma$, then $S^c \subseteq \Sigma \subseteq \Pi$.
        Otherwise, if a point $P \in S^c$ lies outside of $\Sigma$, $\spann{\Sigma}{P}$ would be a (necessarily large) hyperplane, spanned by elements of $S^c$, a contradiction.
        In this way, we see that either some hyperplane contains all points of $S^c$, or for every hyperplane $\Pi$, $S^c \cap \Pi$ is contained in an $(n-3)$-subspace of $\Pi$.
        We can now use the same reasoning to prove that either some hyperplane contains all points of $S^c$, or for every hyperplane $\Pi$, $S^c \cap \Pi$ is contained in an $(n-4)$-space.
        Inductively repeating this process proves the theorem.
        
        Assume that Case $2$ holds. Then there are at most two large hyperplanes, otherwise Case $1$ would hold.
        Consider the set $V = \{(P,\Pi) : P \textnormal{ a point}, \, \Pi \textnormal{ a hyperplane}, \, P \in S\cap\Pi\}$.
        Counting the elements of $V$ in two ways yields
        \[
            |S| \theta_{n-1} = |V| \leqslant 2 \theta_{n-1} + (\theta_n - 2)( 2q^{n-2} + \theta_{n-3} ).
        \]
        Note that the right-hand side equals the exact size of $V$ in case $S$ is the union of two hyperplanes.
        Hence, the right-hand side equals $(2q^{n-1}+\theta_{n-2}) \theta_{n-1}$.
        Thus, $|S| \leqslant 2q^{n-1}+\theta_{n-2}$.
    \end{proof}

	\begin{lemma}\label{Exists2Secant}
		Suppose $\theta_{n-1}<\wt{c}\leqslant\min\big\{(3q-6)\theta_{n-2}+2,B_{n,q}\big\}$. Then there exists a $2$-secant to $\supp{c}$.
	\end{lemma}
	\begin{proof}
		Suppose that no $2$-secant to $\supp{c}$ exists and suppose $t$ is a $q$-secant to $\supp{c}$. By Corollary \ref{PlaneResults}, all planes through $t$ containing at most $2q+1$ points of $\supp{c}$ correspond to planes of type $\difference$. However, such planes contain several $2$-secants, contradicting the assumptions. Thus, by Lemma \ref{Lines2AndQ} and Lemma \ref{ClassificationOf012QQ+1Secants}, all planes through $t$ must contain at least $q^2$ points of $\supp{c}$. In this way,
		\[
			\wt{c}\geqslant\theta_{n-2}\cdot(q^2-q)+q=q^n\textnormal{,}
		\]
		which contradicts the weight assumptions. To conclude, all lines intersect $\supp{c}$ in $0$, $1$ or $q+1$ points, which is only possible if $\supp{c}$ is a subspace. Once again, this contradicts the weight assumptions.
	\end{proof}

	\begin{theorem}\label{TwoHyperplaneTheorem}
		Suppose $\wt{c}\leqslant\min\big\{(3q-6)\theta_{n-2}+2,B_{n,q}\big\}$. Then $c$ is a linear combination of the incidence vectors of at most two distinct hyperplanes.
	\end{theorem}
	\begin{proof}
	    By Theorem \ref{ResultsPolverinoZullo}, we may assume that $2q^{n-1}<\wt{c}\leqslant\min\big\{(3q-6)\theta_{n-2}+2,B_{n,q}\big\}$. The proof will be done by induction on $n$. If $n=2$, Corollary \ref{PlaneResults} finishes the proof. Hence, let $n\geqslant3$ and assume, for each hyperplane $\Pi$, that if $\wt{\restr{c}{\Pi}}\leqslant\min\big\{(3q-6)\theta_{n-3}+2,B_{n-1,q}\big\}$, $\restr{c}{\Pi}$ is a linear combination of at most two distinct $(n-2)$-subspaces of $\Pi$.\\
		Suppose all hyperplanes contain at most $2q^{n-2}+\theta_{n-3}$ points of $\supp{c}$. Since $\supp{c}\neq\emptyset$, there must exist an $(n-2)$-space $\Pi_{n-2}$ intersecting $\supp{c}$ in $q^{n-2}$ or $\theta_{n-2}$ points, such that all hyperplanes through $\Pi_{n-2}$ contain either zero or $q^{n-2}$ points of $\supp{c}\setminus\Pi_{n-2}$. This yields
		\[
			\wt{c}\leqslant\theta_{n-2}+(q+1)q^{n-2} = \theta_{n-1} + q^{n-2} < 2q^{n-1}\textnormal{,}
		\]
		a contradiction.\\
		So consider a hyperplane $\Pi_{n-1}$, containing more than $2q^{n-2}+\theta_{n-3}$ points of $\supp{c}$. Due to Lemma \ref{ClassificationOf012QQ+1Secants}, $\wt{\restr{c}{\Pi_{n-1}}}\geqslant q^{n-1}$ and the holes of $\Pi_{n-1}$ are contained in an $(n-2)$-space $H_{n-2}$ of $\Pi_{n-1}$.
		By Lemma \ref{Exists2Secant}, we find a 2-secant $l$ to $\supp{c}$. Let $P$ and $Q$ be the points in $l \cap \supp{c}$ and let $\alpha=c(P)$.\\
		
		\underline{Case $1$: $P, Q \notin \Pi_{n-1}$.}
		
		Suppose there is at most one hyperplane of type $\difference$ or $\twolines$ through $l$. Fix an $(n-2)$-space $\Pi_{n-2}$ through $l$. By Lemma \ref{ClassificationOf012QQ+1Secants}, at least $q$ hyperplanes through $\Pi_{n-2}$ each contain at least $q^{n-1}$ points of $\supp{c}$, thus
		\[
			\wt{c}\geqslant q^{n-1} + (q-1)\cdot\big(q^{n-1} - \theta_{n-2}\big) = q^n - q^{n-1} + 1\textnormal{,}
		\]
		which exceeds the imposed upper bound on $\wt{c}$ for all prime powers $q$, a contradiction.\\
		Hence, we can choose a $\difference$- or $\twolines$-typed hyperplane $\Sigma_{n-1}$ through $l$, different from the hyperplane $\spann{H_{n-2}}{l}$.
		Therefore, all holes in $\Sigma_{n-1} \cap \Pi_{n-1}$ are contained in the $(n-3)$-space $\Sigma_{n-1}\cap H_{n-2}$.
		As $\supp{\restr{c}{\Sigma_{n-1}}}$ is the union or symmetric difference of precisely two $(n-2)$-subspaces and as $\Sigma_{n-1} \cap \Pi_{n-1}$ must be one of these two, the latter contains either $P$ or $Q$, contrary to the assumption of this case.\\
		
		\underline{Case $2$: $P \in \Pi_{n-1}$.}
		
		Remark that, due to Lemma \ref{ClassificationOf012QQ+1Secants}, $\wt{c}\leqslant2q^{n-1}+\theta_{n-2}$. From this, we get that there are at least $q^{n-2}$ planes through $l$ containing at most $2q+1$ points of $\supp{c}$. Otherwise, we would have
		\[
			2q^{n-1}+\theta_{n-2}\geqslant\wt{c}> q^{n-2}\cdot(2q-2)+(\theta_{n-2}-q^{n-2})q^2+2\textnormal{,}
		\]
		a contradiction whenever $q>2$.\\
		The space $\Pi_{n-1}$ contains $\theta_{n-3}$ planes through a fixed line, so there exists a plane $\pi$ through $l$, not contained in $\Pi_{n-1}$, having at most $2q+1$ points of $\supp{c}$. If $Q \in \Pi_{n-1}$, we could choose another 2-secant lying in such an `external' plane to $\Pi_{n-1}$ and replace $l$ (and $Q$ correspondingly) with this 2-secant. In this way, we may assume that $Q\in\pi\setminus\Pi_{n-1}$. Note that every line through $P$ containing at least two holes of $\Pi_{n-1}$ lies in $H_{n-2}$.
		Therefore, there are at most $\theta_{n-3}$ such lines through $P$. Every plane through $l$ intersects $\Pi_{n-1}$ in a line through $P$, hence there must be at least $q^{n-2} - \theta_{n-3}$ planes through $t$ of type $\difference$ of $\twolines$, resulting in at least $q^{n-2} - \theta_{n-3}$ lines in $\Pi_{n-1}$, through $P$, each containing at least $q$ points all having the same non-zero value $\alpha$ in $c$. This yields at least
		\[
			(q^{n-2} - \theta_{n-3})(q-1) + 1>\frac{1}{2}\theta_{n-1}
		\]
		points in $\Pi_{n-1}$ with value $\alpha$.\\
		Now suppose, on the contrary, that $c$ is a code word of minimal weight such that $c$ cannot be written as a linear combination of at most two hyperplanes. Then $\wt{c-\alpha \Pi_{n-1}}<\wt{c}$, thus the code word $c-\alpha \Pi_{n-1}$ is a linear combination of exactly two hyperplanes. As a consequence, $c$ must be a linear combination of precisely three hyperplanes, implying that $\wt{c}\geqslant3(q^{n-1}-q^{n-2})$, contradicting the weight assumptions.
	\end{proof}
	
	\subsubsection{Going higher on the weight spectrum}
	
	It will turn out that we can go further than the code words of weight $2q^{n-1}+\theta_{n-2}$. Moreover, we will be able to prove that a code word of weight at most $B_{n,q}$ corresponds to a linear combination of hyperplanes through a fixed $(n-3)$-space (Theorem \ref{THETHEOREM}).
	
	\bigskip
	Due to Theorem \ref{TwoHyperplaneTheorem}, we can assume the following on the weight of the code word $c$:
	\[
		(3q-6)\theta_{n-2}+3\leqslant\wt{c}\leqslant B_{n,q}\textnormal{.}
	\]
	As we are mainly interested in the case $n\geqslant3$, the inequality above implies that $q\geqslant29$, which we will keep in mind for the remainder of this section.
		
	\begin{lemma}\label{Exist3Secant}
		Suppose $(3q-6)\theta_{n-2}+3\leqslant\wt{c}\leqslant B_{n,q}$. Then there exists a $3$-secant to $\supp{c}$.
	\end{lemma}
	\begin{proof}
		Suppose that there does not exist a $3$-secant to $\supp{c}$. By Lemma \ref{LinesAndPlanes} and Lemma \ref{q-1Implies3}, all lines intersect $\supp{c}$ in at most $2$ or in at least $q$ points. Applying Lemma \ref{ClassificationOf012QQ+1Secants}, we obtain that $\wt{c}\leqslant2q^{n-1}+\theta_{n-2}$ or $\wt{c}\geqslant q^n$, contradicting our weight assumptions.
	\end{proof}

	\begin{lemma}\label{PlanesThrough3Secant}
		Suppose $(3q-6)\theta_{n-2}+3\leqslant\wt{c} \leqslant B_{n,q}$. Then all planes containing a $3$-secant are planes of type $\boldsymbol{\mathcal{T}}$.
	\end{lemma}
	\begin{proof}
		Suppose that $\sigma$ is a plane of type $\boldsymbol{\mathcal{O}}$ containing a $3$-secant $t$ and suppose that $\Sigma$ is a solid containing $\sigma$. We claim that $\wt{\restr{c}{\Sigma}}\geqslant \frac{1}{4}q^3$.\\
		In the first case, suppose that all planes in $\Sigma$ through $t$ are planes of type $\boldsymbol{\mathcal{O}}$. By Lemma \ref{LinesAndPlanes},
		\begin{align*}
			\wt{\restr c \Sigma}&\geqslant\Big(\frac{1}{2}q^2-\frac{1}{2}q-3\Big)(q+1)+3\\
			&=\frac{1}{2}q^3-\frac{7}{2}q\geqslant\frac{1}{4}q^3\textnormal{,}
		\end{align*}
		the last inequality being valid whenever $q>3$.\\
		In the second case, suppose there exists a plane $\pi$ of type $\boldsymbol{\mathcal{T}}$ in $\Sigma$ through $t$. By Corollary \ref{PlaneResults}, as $\pi$ contains a $3$-secant, $\pi$ is either a plane of type $\odd$, type $\nonconcurrent$ or type $\concurrent$. Regardless of this type, $\pi$ always contains another $3$-secant $t'$ such that $t\cap t'\notin\supp{c}$.\\
		Let $y$ be the number of type-$\boldsymbol{\mathcal{T}}$ planes in $\Sigma$ through $t'$. Remark that such a plane intersects $\sigma$ in at most three points of $\supp{c}$. Indeed, should a $\boldsymbol{\mathcal T}$-typed plane in $\Sigma$ through $t'$ intersect $\sigma$ in at least $4$, thus in at least $q-1$ points (Lemma \ref{LinesAndPlanes}), then one of the three points of $t'\cap\supp{c}$ must lie on this intersection line (as $\pi$ is a plane of type $\boldsymbol{\mathcal{T}}$). But then $t'\cap\sigma\in\supp{c}$, in contradiction with $t\cap t'\notin\supp{c}$. In this way, we get
		\begin{align*}
			\frac{1}{2}q(q-1)\leqslant\wt{\sigma}&\leqslant y\cdot3+(q+1-y)q\\
			&=q^2+q-y(q-3)\textnormal{,}
		\end{align*}
		which implies $y\leqslant\frac{1}{2}(q+7)$, as $q\geqslant29$.\\
		Thus we get that $t'$ is contained in at least $q+1-\frac{1}{2}(q+7)=\frac{1}{2}(q-5)$ planes of type $\boldsymbol{\mathcal{O}}$ (all lying in $\Sigma$). As each $\boldsymbol{\mathcal{T}}$-typed plane in $\Sigma$ through $t'$ contains at least $3q-3$ points of $\supp{c}$, we get
		\begin{align*}
			\wt{\restr c \sigma}&\geqslant\bigg\lceil\frac{1}{2}(q-5)\bigg\rceil\cdot\Big(\frac{1}{2}q(q-1)-3\Big)+\bigg\lfloor\frac{1}{2}(q+7)\bigg\rfloor\cdot(3q-3-3)+3\\
			&\geqslant\Big(\frac{1}{2}(q-5)\Big)\cdot\Big(\frac{1}{2}q(q-1)-3\Big)+\Big(\frac{1}{2}(q+6)\Big)\cdot(3q-3-3)+3\\
			&=\frac{1}{4}q^3+\frac{23}{4}q-\frac{15}{2}\geqslant\frac{1}{4}q^3\textnormal{.}
		\end{align*}
		As the above claim holds for all solids containing $\sigma$, we get
		\[
			\wt{c}\geqslant\theta_{n-3}\Big(\frac{1}{4}q^3\Big)-(\theta_{n-3}-1)(q^2+q+1)\textnormal{.}
		\]
		One can easily check this implies $\wt{c}\geqslant B_{n,q}$ for all prime powers $q$, a contradiction.
	\end{proof}
	
	We can generalise the above lemma, which will prove its usefulness when using induction.
	
	\begin{lemma}\label{SpacesThrough3Secant}
		Suppose $(3q-6)\theta_{n-2}+3\leqslant\wt{c}\leqslant B_{n,q}$. Let $\psi$ be a $k$-space, $2\leqslant k<n$, containing a $3$-secant $s$. Then $\wt{\restr{c}{\psi}}\leqslant B_{k,q}$.
	\end{lemma}
	\begin{proof}
		By Lemma \ref{PlanesThrough3Secant}, we know that all planes in $\psi$ through $t$ contain at most $3q+1$ points of $\supp{c}$ (Corollary \ref{PlaneResults}). This implies that $\wt{\restr c \psi}\leqslant\theta_{k-2}(3q+1-3)+3\leqslant B_{k,q}$.
	\end{proof}
	
	Remark that the last inequality in the above proof is the reason why the bound $B_{n,q}$ differs in value for $q\in\{29,31,32\}$.
    
    \bigskip
	We can now present some properties about certain types of subspaces sharing a common $3$-secant.
	
	\begin{lemma}\label{OneTypeCat}
		Suppose $(3q-6)\theta_{n-2}+3\leqslant\wt{c}\leqslant B_{n,q}$. Let $\Pi_1$ and $\Pi_2$ be two $k$-spaces, $2\leqslant k<n$, of type $T_1,T_2\in\{\odd,\nonconcurrent,\concurrent\}$, respectively, having a $3$-secant $s$ in common. Then at least one of the following holds:
		\begin{enumerate}
		    \item[\textnormal{(1)}] $T_1=\concurrent$.
		    \item[\textnormal{(2)}] $T_2=\concurrent$.
		    \item[\textnormal{(3)}] $T_1=T_2$.
		\end{enumerate}
		Furthermore, if $T_1=T_2$, then $\wt{\restr{c}{\Pi_1}}=\wt{\restr{c}{\Pi_2}}$.
	\end{lemma}
	\begin{proof}
	    In each subspace $\Pi_i$, choose a plane $\pi_i$ through $s$, disjoint to the vertex corresponding with the cone $\supp{\restr{c}{\Pi_i}}$. By definition, $\pi_i$ is a plane of type $T_i$. Define $\Sigma=\spann{\pi_1}{\pi_2}$.\\
		Furthermore, let $P_\alpha$, $P_\beta$ and $P_\gamma$ be the points in $s\cap\supp{c}$ with corresponding non-zero values $\alpha$, $\beta$ and $\gamma$ in $c$. Let $l_\alpha^{(i)}$, $l_\beta^{(i)}$ and $l_\gamma^{(i)}$ be the unique long secants in $\pi_i$ through $P_\alpha$, $P_\beta$ and $P_\gamma$, respectively ($i=1,2$).\\
		
		\underline{Case $1$: $T_1=\odd$ and $T_2=\nonconcurrent$.}
		
		Suppose that $\pi$ is a plane in $\Sigma$ going through $l_\alpha^{(2)}$. Remark that $l_\alpha^{(2)}$ is a long secant, containing $q-1$ points having non-zero value $\alpha$, one point having value $\alpha+\beta$ and one point having value $\alpha+\gamma$. From this, we know that the plane $\pi$ \emph{cannot} be
		\begin{itemize}
			\item a plane of type $\zero$, as $\alpha\neq0$.
			\item a plane of type $\codline$, $\difference$, $\twolines$ or $\concurrent$, else $\alpha+\beta=\alpha$ or $\alpha+\gamma=\alpha$.
			\item a plane of type $\odd$, as $l_\alpha^{(2)}$ contains at least three points with the same value $\alpha$.
			\item a plane of type $\nonconcurrent$, unless $\wt{\restr{c}{\pi}}=\wt{\restr{c}{\pi_2}}$. Indeed, $l_\alpha^{(2)}$ contains two points $l_\alpha^{(2)}\cap l_\beta^{(2)}$ and $l_\alpha^{(2)}\cap l_\gamma^{(2)}$ with corresponding values $\alpha+\beta$ and $\alpha+\gamma$, respectively, unambiguously fixing the weight of $\wt{\restr{c}{\pi}}$.
		\end{itemize}
		However, $\pi$ can only be a plane of type $\nonconcurrent$ in some cases. Suppose that $\pi$ is a plane of type $\nonconcurrent$ and suppose that $\pi$ intersects $\pi_1$ in a $3$-secant $t$. One of the points of $t\cap\supp{c}$ is obviously $P_\alpha$, as this point belongs to both $l_\alpha^{(2)}$ and $\pi_1$. The other two points of $t\cap\supp{c}$ lie on $l_\beta^{(1)}$ and $l_\gamma^{(1)}$ and must have corresponding values $\beta$ and $\gamma$, as $\wt{\restr{c}{\pi}}=\wt{\restr{c}{\pi_2}}$. As $\pi_1$ is a plane of type $\odd$, there are only two possibilities for $\pi$ to intersect $\pi_1$, namely when the $\beta$-valued point of $t$ lies on $l_\beta^{(1)}$ (then $\pi=\pi_2$), or when the $\beta$-valued point of $t$ lies on $l_\gamma^{(1)}$. Conclusion: of the at least $q-2$ planes through $l_\alpha^{(2)}$ in $\Sigma$, intersecting $\pi_1$ in a $3$-secant, at least $q-4$ of them cannot be a plane of type $\nonconcurrent$, and thus must be planes of type $\boldsymbol{\mathcal{O}}$. In addition, the plane $\spann{l_\alpha^{(1)}}{l_\alpha^{(2)}}$ can never be a plane of type $\nonconcurrent$ as well, as $l_\alpha^{(1)}$ contains many distinctly valued points. Thus, we find at least $q-3$ planes of type $\boldsymbol{\mathcal{O}}$ in $\Sigma$ through $l_\alpha^{(2)}$, each containing at least $\frac{1}{2}q(q-1)$ points of $\supp{c}$ (Lemma \ref{LinesAndPlanes}). The other planes in $\Sigma$ through $l_\alpha^{(2)}$, of which there are at most four, contain at least $3q-3$ points of $\supp{c}$. We get
		\begin{align*}
			\wt{\restr{c}{\Sigma}}&\geqslant\Big(\frac{1}{2}q(q-1)\Big)(q-3)+4\cdot(3q-3)-q\cdot(q+1)\\
			&=\frac{1}{2}q^3-3q^2+\frac{25}{2}q-12>B_{3,q}\textnormal{,}
		\end{align*}
		which is, if $n=3$, a direct contradiction or, if $n>3$, a contradiction with Lemma \ref{SpacesThrough3Secant}, as $\Sigma$ contains the $3$-secant $s$.\\
		
		\underline{Case $2$: $T_1=T_2$.}
		
		Suppose, on the contrary, that $\wt{\restr{c}{\Pi_1}}\neq\wt{\restr{c}{\Pi_2}}$. W.l.o.g. we can assume that $\wt{\restr{c}{\pi_1}}\neq\wt{\restr{c}{\pi_2}}$ as well. Assume, in the first case, that $T_1\in\{\odd,\concurrent\}$. By observing the types of these planes and by Proposition \ref{PropOddCodeword}, $\wt{\restr{c}{\pi_1}}\neq\wt{\restr{c}{\pi_2}}$ implies that both $\alpha+\beta+\gamma=0$ and $\alpha+\beta+\gamma\neq0$, a contradiction.\\
	    Now assume $T_1=\nonconcurrent$. Considering the plane $\pi_i$, we know that the lines $l_\alpha^{(i)}$, $l_\beta^{(i)}$ and $l_\gamma^{(i)}$ are not concurrent. As $\wt{\restr{c}{\pi_1}}\neq\wt{\restr{c}{\pi_2}}$, we know, without loss of generality, that the value of the point $l_\alpha^{(1)}\cap l_\beta^{(1)}$ is zero, while the value of the point $l_\alpha^{(2)}\cap l_\beta^{(2)}$ is not zero. This implies that $\alpha+\beta$ is both zero and non-zero, a contradiction.
	\end{proof}

	\begin{lemma}\label{UniqueLongSecants}
		Suppose $q>3$ and let $\pi$ be a plane of type $T\in\{\odd,\nonconcurrent\}$. Then all planes $\sigma$ of type $\mathcal{T}$ intersecting $\pi$ in a long secant are planes of type $T$ as well. Moreover, $\wt{\restr{c}{\sigma}}=\wt{\restr{c}{\pi}}$.
	\end{lemma}
	\begin{proof}
		Suppose the plane $\sigma$ is a plane of type $T_\sigma\in\mathcal{T}$; let $l$ be the long secant $\pi\cap\sigma$. As $T\in\{\odd,\nonconcurrent\}$, no $q$ points on $l$ have the same non-zero value in $c$. As a consequence, $T_\sigma\notin\{\zero,\codline,\difference,\twolines,\concurrent\}$. If $T=\odd$, we find at least $q$ points on $l$ having pairwise different values in $c$. If $T=\nonconcurrent$, we find at most $3$ different points on $l$ having pairwise different values. Hence, if $T_\sigma\neq T$, then $q\leqslant3$, a contradiction. Furthermore, it is not hard to check that the set of values of points on $l$ fixes the weight of $\restr{c}{\sigma}$.
	\end{proof}

	\begin{lemma}\label{NotAllCatBPlane}
		Suppose that $n=3$ and $3q^2-3q-3\leqslant\wt{c}\leqslant B_{3,q}$. Then a $3$-secant is never contained in $q+1$ planes of the same type $T\in\{\odd,\nonconcurrent\}$.
	\end{lemma}
	\begin{proof}
		Suppose, on the contrary, that $t$ is such a $3$-secant. Fix a plane $\pi$ through $t$. By Lemma \ref{OneTypeCat}, the weight of the code word $c$ is known, as we can count: $\wt{c}=(q+1)\big(\wt{\restr{c}{\pi}}-3\big)+3=(q+1)\wt{\restr{c}{\pi}}-3q$.\\
		Remark that, as $\pi$ is a plane of either type $\odd$ or $\nonconcurrent$, we can always find a $1$- or $2$-secant $r$ in $\pi$ such that $t$ and $r$ intersect in a point $Q$ of $\supp{c}$. Indeed,
		\begin{itemize}
			\item if $\pi$ is a plane of type $\odd$, we can simply connect two points: a hole lying on a long secant in $\pi$, different from the intersection point of the three long secants in $\pi$, with a point of $t\cap\supp{c}$ on another long secant in $\pi$.
			\item if $\pi$ is plane of type $\nonconcurrent$, we can connect a point lying on two long secants with the unique point of $t$ lying on the third long secant.
		\end{itemize}
		Let $\sigma$ be a plane through $r$, not equal to $\pi$. Choose a long secant $s$ in $\sigma$ through $Q$. This is possible since every plane of type $\mathcal T \setminus \{T_0\}$ obviously contains a long secant, and planes of type $\mathcal O$ contain long secants as well (cfr. Lemma \ref{LinesAndPlanes}).
		The plane $\spann{t}{s}$ contains the $3$-secant $t$, thus this plane has to be of the same subtype as $\pi$.
		In particular, this means that $\spann{t}{s}$ is a plane of type $\odd$ or $\nonconcurrent$.
		However, by Lemma \ref{UniqueLongSecants}, the plane $\spann{t}{s}$ then has to be of the same type as $\sigma$ as well, as they share the long secant $s$, unless $\sigma$ is a plane of type $\boldsymbol{\mathcal{O}}$.\\
		Therefore, all planes $\sigma$ through $r$ satisfy either $\wt{\restr{c}{\sigma}} = \wt{\restr{c}{\pi}}$ (if $\sigma$ is a plane of type $\mathcal T$), or $\wt{\restr{c}{\sigma}} \geq \frac{1}{2}q(q-1) > \wt{\restr{c}{\pi}}$ (if $\sigma$ is a plane of type $\mathcal O$, by Lemma \ref{LinesAndPlanes}). 
		In both cases, this yields the following lower bound on $\wt{c}$:
		\[
			(q+1)\wt{\restr{c}{\pi}}-3q=\wt{c}\geqslant(q+1)\big(\wt{\restr{c}{\pi}}-2\big)+2\textnormal{,}
		\]
		a contradiction.
	\end{proof}

	The following proposition is a consequence of the way code words of type $\boldsymbol{\mathcal{T}}$ are defined (Definition \ref{DefHyperplane}).

	\begin{proposition}\label{TypesOfPlanesAndHyperplanes}
		Suppose that $\Pi$ is a hyperplane of type $T\in\boldsymbol{\mathcal{T}}$, with $\kappa$ the $(n-4)$-dimensional vertex of $\supp{\restr{c}{\Pi}}$. Suppose that $t$ is a $3$-secant contained in $\Pi$. Then $t$ is disjoint to $\kappa$ and all $q^{n-3}$ planes in $\Pi$ that contain $t$ but that are disjoint to $\kappa$ are planes of type $T$. The other $\theta_{n-4}$ planes in $\Pi$ through $t$ intersect $\kappa$ in a point and are all planes of type $\concurrent$.
	\end{proposition}

	\begin{lemma}\label{NotAllCatBHyper}
		Suppose that $(3q-6)\theta_{n-2}+3\leqslant\wt{c}\leqslant B_{n,q}$. Then a $3$-secant is never contained in $\theta_{n-2}$ hyperplanes of the same type $T\in\{\odd,\nonconcurrent\}$.
	\end{lemma}
	\begin{proof}
		By Lemma \ref{NotAllCatBPlane}, we can assume that $n>3$. Suppose that $t$ is a $3$-secant with the described property. Now define
		\[
			S=\{(\pi,\Pi):t\subseteq\pi\subseteq\Pi, \pi\textnormal{ a plane},\Pi\textnormal{ a hyperplane},\textnormal{both of type }T\}\textnormal{.}
		\]
		Fix an arbitrary $T$-typed plane $\pi_0\supseteq t$. As all hyperplanes through $t$ are of the same type $T$, all hyperplanes through $\pi_0$ have this property as well. Thus, the number of elements in $S$ with a fixed first argument $\pi_0$ equals $\theta_{n-3}$.\\
		Fix an arbitrary $T$-typed hyperplane $\Pi_0\supseteq t$. By Proposition \ref{TypesOfPlanesAndHyperplanes}, the number of elements in $S$ with a fixed second argument $\Pi_0$ equals $q^{n-3}$ (the number of planes in $\Pi_0$ through $t$, disjoint to an $(n-4)$-subspace not intersecting $t$).\\
		Let $x_\pi$ be the number of $T$-typed planes through $t$. By double counting, we get:
		\[
			x_\pi\cdot\theta_{n-3}=|S|=\theta_{n-2}\cdot q^{n-3}\quad\Longleftrightarrow\quad x_\pi=\frac{q^{n-1}-1}{q^{n-2}-1}q^{n-3}=q^{n-2}+1-\frac{q^{n-3}-1}{q^{n-2}-1}
		\]
		As $x_\pi$ is known to be an integer, this is only valid when the fraction on the right is an integer. As $n>3$, this is never the case.
	\end{proof}

	\begin{lemma}\label{VerticesAreEqual}
		Suppose that $(3q-6)\theta_{n-2}+3\leqslant\wt{c}\leqslant B_{n,q}$. Let $\Pi_1$ and $\Pi_2$ be two hyperplanes of type $\concurrent$ and let $\mathcal{C}_i$ be the union of the three $(n-2)$-subspaces present in the linear combination $\restr{c}{\Pi_i}$, thus intersecting in a common $(n-3)$-space $\kappa_i$ ($i=1,2$). Suppose that $\mathcal{C}_1$ and $\mathcal{C}_2$ have an $(n-2)$-subspace in common. Then either
		\begin{itemize}
		    \item $\kappa_1=\kappa_2$, or
		    \item $n>3$ and there exists a solid $S$ containing a long secant that is only contained in planes in $S$ of type $\mathcal{T}$.
		\end{itemize}
	\end{lemma}
	\begin{proof}
		Let $\Sigma$ be the $(n-2)$-space that $\mathcal{C}_1$ and $\mathcal{C}_2$ have in common. As $q>3$, $\Sigma$ must be one of the three subspaces present in the linear combination of both $\restr{c}{\Pi_1}$ and $\restr{c}{\Pi_2}$.\\
		Suppose that $\kappa_1\neq\kappa_2$. As these are spaces of the same dimension, we can find a point $P_1\in\kappa_1\setminus\kappa_2$ and a point $P_2\in\kappa_2\setminus\kappa_1$; define $l=\spann{P_1}{P_2}$.
		Remark that $l$ must be a $(q+1)$-secant to $\supp{c}$.
		This follows from the fact that every point of $l$ lies in $\Sigma \setminus \kappa_i$, for at least one choice of $i$.
		Looking in $\mathcal C_i$, we see that all points of $\Sigma \setminus \kappa_i$ lie in $\supp c$.
		Now take planes $\pi_i$ in $\Pi_i$, for $i = 1, 2$, through $l$, not contained in $\Sigma$. Due to this choice, it is clear that the plane $\pi_i$ will intersect each $(n-2)$-subspace of $\mathcal{C}_i$ in a line (through $P_i$). Define $S=\spann{\pi_1}{\pi_2}$.\\
		Choose a $(q+1)$-secant $s$ in $\pi_1$, different from $l$. As $P_1\neq P_2$, all planes in $S$ through $s$ (not equal to $\pi_1$) intersect $\pi_2$ in a $3$-secant and thus, by Lemma \ref{PlanesThrough3Secant}, are planes of type $\boldsymbol{\mathcal{T}}$. As $\pi_1$ is a plane of type $\boldsymbol{\mathcal{T}}$ as well, we know that all planes in $S$ through $s$ are planes of type $\boldsymbol{\mathcal{T}}$.\\
		If $n=3$, we get that $(3q-6)(q+1)+3\leqslant\wt{c}=\wt{\restr{c}{S}}\leqslant q\cdot(2q)+(3q+1)=2q^2+3q+1$, which is only valid if $q<7$, contrary to the assumptions.
	\end{proof}
	
	The following theorem connects all previous results and proves Theorem \ref{THETHEOREM}.

	\begin{theorem}\label{LastTheorem}
		Suppose $(3q-6)\theta_{n-2}+3\leqslant\wt{c}\leqslant B_{n,q}$. Then there exists a plane $\pi$ of type $T\in\{\odd,\nonconcurrent,\concurrent\}$ and an $(n-3)$-space $\kappa$ such that $\pi\cap\kappa=\emptyset$ and
		\[
		    c=\sum_{P\in\pi}c(P)\cdot\spann{\kappa}{P}\textnormal{.}
		\]
	\end{theorem}
	\begin{proof}
		We will prove this by induction on $n$. When $n = 2$, we can choose $\kappa=\emptyset$ and refer to Corollary \ref{PlaneResults}. Now assume $n\geqslant3$ and suppose the statement is true for $c$ restricted to any $k$-space, $2\leqslant k<n$. By Lemma \ref{Exist3Secant}, we can choose a $3$-secant $t$ with corresponding non-zero values $\alpha$, $\beta$ and $\gamma$. By the induction hypothesis and Lemma \ref{SpacesThrough3Secant}, each hyperplane through $t$ is a hyperplane of type $\boldsymbol{\mathcal{T}}$ and by Lemma \ref{OneTypeCat}, we know that there exist two specific types $T_A=\concurrent$ and $T_B\in\{\odd,\nonconcurrent,\concurrent\}$ such that all hyperplanes through $t$ are either of type $T_A$ or type $T_B$.
		Furthermore, by Lemma \ref{NotAllCatBHyper}, we know that there exists at least one hyperplane through $t$ of type $T_A$; consider such a hyperplane $\Pi$. Remark that by Proposition \ref{TypesOfPlanesAndHyperplanes}, all planes through $t$ are planes of type $T_A$ or $T_B$ as well. We can now fix a certain plane $\pi$ as follows: if all planes through $t$ are planes of type $T_A$, choose $\pi$ to be an arbitrary plane through $t$, not contained in $\Pi$. Else, choose $\pi$ to be a plane through $t$ of type $T_B$.
		By Proposition \ref{TypesOfPlanesAndHyperplanes}, $\pi$ cannot be contained in $\Pi$.\\
		Furthermore, we know that $\restr{c}{\Pi}$ is a linear combination of three different $(n-2)$-subspaces of $\Pi$ through an $(n-3)$-space. Choose $\kappa$ to be this $(n-3)$-space. As all lines in $\Pi$, not disjoint to $\kappa$, are either $0$-, $1$-, $q$- or $(q+1)$-secants, we know that $\kappa$ must be disjoint to the $3$-secant $t$ and, furthermore, disjoint from the plane $\pi\supseteq t$, as that plane is not contained in $\Pi$.
		
		\bigskip
		For each point $P\in\kappa$, it is easy to see that $c(P)$ is equal to the sum of the values of the points on the $3$-secant $t$ (which is $\alpha+\beta+\gamma$).\\
		As $\restr{c}{\Pi}$ is a linear combination of three different $(n-2)$-spaces of $\Pi$ having the space $\kappa$ in common, we can choose one of those $(n-2)$-spaces $\Psi_1$; w.l.o.g. this space corresponds to the value $\alpha$. Choose an arbitrary $3$-secant $t_1$ in $\pi$ through the point $\Psi_1\cap\pi$, thus having corresponding non-zero values $\alpha$, $\beta_1$ and $\gamma_1$. By the induction hypothesis and Lemma \ref{SpacesThrough3Secant}, $\Pi_1=\spann{\Psi_1}{t_1}$ is a hyperplane of type $\boldsymbol{\mathcal{T}}$. We claim that $\Pi_1$ is a hyperplane of type $T_A$. Indeed, let $\pi_1$ be a plane in $\Pi_1$ through $t_1$, thus intersecting $\Pi$ in a line of $\Psi_1$. Then this intersection line must be a $q$- or $(q+1)$-secant. By Lemma \ref{PlanesThrough3Secant}, $\pi_1$ has to be a plane of type $\boldsymbol{\mathcal{T}}$ and, more specifically, a plane of type $T_A$ (Lemma \ref{UniqueLongSecants}). As such, all planes in $\Pi_1$ through $t_1$ are planes of type $T_A$, thus $\Pi_1$ contains at least $\theta_{n-3}$ planes of type $T_A$ through a fixed $3$-secant ($t_1$). By Proposition \ref{TypesOfPlanesAndHyperplanes}, at least one of these planes is of the same type as $\Pi_1$, thus this hyperplane must be of type $T_A$.\\
		Let $\kappa_1$ be the $(n-3)$-subspace of $\Pi_1$ in which the three hyperplanes of $\restr{c}{\Pi_1}$ intersect. By Lemma \ref{VerticesAreEqual}, we know that $\kappa=\kappa_1$. In this way, it is easy to see that all points in $\Pi_1\setminus\kappa$ fulfil the desired property.\\
		We can now repeat the above process by choosing another $(n-2)$-space $\Psi_2$ in one of the linear combinations of $\restr{c}{\Pi}$ or $\restr{c}{\Pi_1}$ and considering the span $\Pi_2=\spann{\Psi_2}{t_2}$, with $t_2$ an arbitrary $3$-secant in $\pi$ through the point $\Psi_2\cap\pi$. All points in $\Pi_2\setminus\kappa$ will fulfil the desired property as well.\\
		To conclude, if, for each point $P$ in $\pi$, there exists a sequence of $3$-secants $t_1,t_2,\dots,t_m \ni P$ in $\pi$ such that $t\cap t_1\in\supp{c}$ and $t_i\cap t_{i+1}\in\supp{c}$ for all $i\in\{1,2,\dots,m-1\}$, then this theorem is proven by consecutively repeating the above arguments. Unfortunately, not all points in $\pi$ satisfy this property. However, if a point $P\in\pi$ does not lie on such a (sequence of) $3$-secant(s), we can easily prove it lies on a $0$-, $1$- or $2$-secant $r$ in $\pi$ having $q$ points that do satisfy this first property. 
		Thus, we already know the value of a lot of points in the hyperplane $\spann{\kappa}{r}$, namely of precisely $|\spann{\kappa}{r}|-|\spann{\kappa}{P}|+|\kappa|=\theta_{n-1}-q^{n-2}$ points. Furthermore, $\wt{\restr{c}{\spann{\kappa}{r}}}\leqslant2q^{n-2}+\theta_{n-3}+\wt{\restr{c}{\spann{\kappa}{P}}}-\wt{\restr{c}{\kappa}}\leqslant3q^{n-2}+\theta_{n-3}\leqslant B_{n-1,q}$. Thus, by the induction hypothesis, this hyperplane is a hyperplane of type $\boldsymbol{\mathcal{T}}$. It is easy to see that all points in $\spann{\kappa}{P}$ must satisfy the property of the theorem.
	\end{proof}
	
	\textbf{Acknowledgement.} Special thanks to Maarten De Boeck for revising these results with great care and eye for detail.
	
	\appendix
	
	\section{Further details to Lemma \ref{LinesAndPlanes}}\label{Appendix}
	
	Suppose $c\in C_{n-1}(n,q)$, with $q\geqslant7$, $q\notin\{8,9,16,25,27,49\}$, and assume that $\wt{c}\leqslant D_{n,q}$, with
	\[
	    D_{n,q}=\begin{cases}
		    \Big(3q-\sqrt{6q}-\frac{1}{2}\Big)q^{n-2}\;&\textnormal{if }q\in\{7,11,13,17\}\textnormal{,}\\
		    \Big(3q-\sqrt{6q}+\frac{9}{2}\Big)q^{n-2}\;&\textnormal{if }q\in\{19,121\}\textnormal{,}\\
		    \Big(4q-\sqrt{8q}-\frac{33}{2}\Big)q^{n-2}\;&\textnormal{otherwise;}\\
		\end{cases}
	    A_q=\begin{cases}
	        3q-3\;&\textnormal{if }q\in\{7,11,13,17\}\textnormal{,}\\
		    3q+2\;&\textnormal{if }q\in\{19,121\}\textnormal{,}\\
		    4q-21\;&\textnormal{otherwise.}\\
	    \end{cases}
	\]
	Remark that $B_{n,q}<D_{n,q}$ if $q\in\{29,31,32\}$ and $B_{n,q}=D_{n,q}$ for all other considered values of $q$, so it suffices to check the details of the lemma for this bound $D_{n,q}$. We will prove a contradiction using the following two inequalities:
	\begin{equation}\label{inequalities}
	    \wt{c}\geqslant\bigg(\frac{1}{2}j(j+1)-j\bigg)\theta_{n-2}+j\qquad\textnormal{and}\qquad j\geqslant\frac{A_q\theta_{n-2}-\wt{c}}{\theta_{n-2}-1}\textnormal{.}
	\end{equation}
	Define $\boldsymbol{W}:=\wt{c}$. Below, we will sketch the details when $q>17$, $q\notin\{25,27,49\}$. The other two cases are completely analogous.\\
	Combining the two equations in \eqref{inequalities}, knowing that $A_q=4q-21$, gives rise to the following inequality:
	\begin{align*}
	    0\geqslant(&q^{n+1}-2q^n+q^{n-1}-q^2+2q-1)\boldsymbol{W}^2\\
	    -&(8q^{2n}-49q^{2n-1}+41q^{2n-2}-17q^{n+1}+100q^n-83q^{n-1}+9q^2-51q+42)\boldsymbol{W}\\
	    +&16q^{3n-1}-172q^{3n-2}+462q^{3n-3}-36q^{2n}+441q^{2n-1}-1323q^{2n-2}\\
	    -&8q^{n+2}+82q^{n+1}-458q^n+1302q^{n-1}+8q^3-62q^2+189q-441
	\end{align*}
	The above inequality is of the form $0\geqslant a\boldsymbol{W}^2+b\boldsymbol{W}+c$, with $a\geqslant0$, implying that $\boldsymbol{W}\geqslant\frac{-b-\sqrt{D}}{2a}$ with $D=\sqrt{b^2-4ac}$. One can check that
	\begin{align*}
	    D^2=32q^{4n-1}-&231q^{4n-2}+366q^{4n-3}-167q^{4n-4}\\
	    -&64q^{3n+1}+398q^{3n}-270q^{3n-1}-398q^{3n-2}+334q^{3n-3}\\
	    +&32q^{2n+3}-103q^{2n+2}-526q^{2n+1}+1066q^{2n}-302q^{2n-1}-167q^{2n-2}\\
	    -&64q^{n+4}+398q^{n+3}-270q^{n+2}-398q^{n+1}+334q^n\\
	    +&32q^5-231q^4+366q^3-167q^2
	\end{align*}
	Keeping in mind that $q\geqslant23$, we can raise the right-hand side and obtain
	\begin{equation}\label{inequality1}
	    D^2\leqslant32q^{4n-1}-231q^{4n-2}+398q^{4n-3}-46q^{3n+1}\textnormal{.}
	\end{equation}
	On the other hand, we have that $D^2\geqslant \big(-b-2a(4q-\sqrt{8q}-\frac{33}{2})\big)^2$, which implies
	\begin{align*}
	    D^2\geqslant32q^{4n-1}-&128q^{4n-2}-264\sqrt{2q}\cdot q^{4n-3}+192q^{4n-3}+792\sqrt{2q}\cdot q^{4n-4}+961q^{4n-4}\\
	    -&792\sqrt{2q}\cdot q^{4n-5}-2146q^{4n-5}+264\sqrt{2q}\cdot q^{4n-6}+1089q^{4n-6}\\
	    -&72\sqrt{2q}\cdot q^{3n}-64q^{3n}+552\sqrt{2q}\cdot q^{3n-1}+850q^{3n-1}-696\sqrt{2q}\cdot q^{3n-2}\\
	    -&4344q^{3n-2}-504\sqrt{2q}\cdot q^{3n-3}+4216q^{3n-3}+1248\sqrt{2q}\cdot q^{3n-4}+1520q^{3n-4}\\
	    -&528\sqrt{2q}\cdot q^{3n-5}-2178q^{3n-5}+81q^{2n+2}+144\sqrt{2q}\cdot q^{2n+1}-886q^{2n+1}\\
	    -&1104\sqrt{2q}\cdot q^{2n}+2041q^{2n}+2184\sqrt{2q}\cdot q^{2n-1}+3828q^{2n-1}-1368\sqrt{2q}\cdot q^{2n-2}\\
	    -&9551q^{2n-2}-120\sqrt{2q}\cdot q^{2n-3}+3398q^{2n-3}+264\sqrt{2q}\cdot q^{2n-4}+1089q^{2n-4}\\
	    -&162q^{n+3}-72\sqrt{2q}\cdot q^{n+2}+1836q^{n+2}+552\sqrt{2q}\cdot q^{n+1}-6120q^{n+1}\\
	    -&1224\sqrt{2q}\cdot q^n+4608q^n+1080\sqrt{2q}\cdot q^{n-1}+2610q^{n-1}-336\sqrt{2q}\cdot q^{n-2}\\
	    -&2772q^{n-2}+81q^4-918q^3+3357q^2-4284q+1764\textnormal{.}
	\end{align*}
	Keeping in mind that $q\geqslant23$, we can lower the right-hand side and obtain
	\begin{equation}\label{inequality2}
	    D^2\geqslant32q^{4n-1}-206q^{4n-2}-72\sqrt{2q}\cdot q^{3n}-64q^{3n}\textnormal{.}
	\end{equation}
	Combining \eqref{inequality1} and \eqref{inequality2}, we obtain
	\[
	    32q^{4n-1}-231q^{4n-2}+398q^{4n-3}-46q^{3n+1}\geqslant D^2\geqslant 32q^{4n-1}-206q^{4n-2}-72\sqrt{2q}\cdot q^{3n}-64q^{3n}\textnormal{,}
	\]
	resulting in
	\begin{align*}
	    &0\geqslant25q^{4n-2}-398q^{4n-3}+46q^{3n+1}-72\sqrt{2q}\cdot q^{3n}-64q^{3n}\\
	    \Longrightarrow\quad&0\geqslant25q^{4n-2}-398q^{4n-3}\\
	    \Longrightarrow\quad&\frac{398}{25}\geqslant q\textnormal{,}
	\end{align*}
	a contradiction.

	Authors address:\\
	Sam Adriaensen\\
	Vrije Universiteit Brussel, Department of Mathematics\\
	Pleinlaan $2$\\
	$1050$ Brussels\\
	BELGIUM\\
	\texttt{e-mail: sam.adriaensen@vub.be}
	
	\bigskip
	Lins Denaux, Leo Storme\\
	Ghent University
	Department of Mathematics: Analysis, Logic and Discrete Mathematics\\
	Krijgslaan $281$ -- Building S$8$\\
	$9000$ Ghent\\
	BELGIUM\\
	\texttt{e-mail: lins.denaux@ugent.be}\\
	\texttt{e-mail: leo.storme@ugent.be}
	
	\bigskip
	Zsusza Weiner\\
	MTA-ELTE Geometric and Algebraic Combinatorics Research Group\\
	H-$1117$ Budapest, P\'azm\'any P\'eter s\'et\'any $1$/C\\
	Prezi.com\\
	H-$1065$ Budapest, Nagymez\H o utca $54$--$56$\\
	HUNGARY\\
	\texttt{e-mail: zsuzsa.weiner@gmail.com}

\end{document}